\documentclass[12pt]{amsart}

\usepackage{amsmath}
\usepackage{amssymb}
\usepackage{amsfonts}
\usepackage{amsthm}
\usepackage{enumerate}
\usepackage{hyperref}
\usepackage{color}

\theoremstyle{plain}
\newtheorem{thm}{Theorem}[section]

\newtheorem{prop}[thm]{Proposition}
\newtheorem{lem}[thm]{Lemma}
\newtheorem{cor}[thm]{Corollary}

\newtheorem{ques}[thm]{Question}

\theoremstyle{definition}
\newtheorem{dfn}[thm]{Definition}

\newtheorem{rems}[thm]{Remarks}
\newtheorem{dfns-rems}[thm]{Definitions and Remarks}
\newtheorem{notas-rems}[thm]{Notations and Remarks}
\newtheorem{exmps-rems}[thm]{Examples and Remarks}

\begin{document}


\title[Symbolic powers of cover ideals]{On the minimal free resolution of symbolic powers of cover ideals of graphs}


\author[S. A. Seyed Fakhari]{S. A. Seyed Fakhari}

\address{S. A. Seyed Fakhari, School of Mathematics, Statistics and Computer Science,
College of Science, University of Tehran, Tehran, Iran.}

\email{aminfakhari@ut.ac.ir}


\begin{abstract}
For any graph $G$, assume that $J(G)$ is the cover ideal of $G$. Let $J(G)^{(k)}$ denote the $k$th symbolic power of $J(G)$. We characterize all graphs $G$ with the property that $J(G)^{(k)}$ has a linear resolution for some (equivalently, for all) integer $k\geq 2$. Moreover, it is shown that for any graph $G$, the sequence $\big({\rm reg}(J(G)^{(k)})\big)_{k=1}^{\infty}$ is nondecreasing. Furthermore, we compute the largest degree of minimal generators of $J(G)^{(k)}$ when $G$ is either an unmixed of a claw-free graph.
\end{abstract}






\thanks{}


\maketitle


\section{Introduction} \label{sec1}

The study of minimal free resolution of monomial ideals and their powers is an interesting and
active topic of mathematics which employs methods of combinatorics in algebraic context. One of the main result in this area was obtained by
Fr${\rm \ddot{o}}$berg \cite[Theorem 1]{f}, who characterized all quadratic squarefree monomial ideals which have a linear resolution. Herzog, Hibi and Zheng \cite{hhz} proved that a quadratic squarefree monomial ideal $I$ has a linear resolution if and only if every power of $I$ has a linear resolution. It is also known \cite{ht} that polymatroidal ideals have a linear resolution. On the other hand, the powers
of polymatroidal ideals are again polymatroidal (see \cite{hh} and hence,
they have a linear resolution. However, it is not in general true that if a monomial ideal $I$ has a linear resolution, then its powers have the same property. Examples of this kind of ideals were provided by Terai (see \cite[remark 3]{c}) and Sturmfels \cite{s1}.

In this paper, we focus on powers of unmixed squarefree monomial ideals of height two. These ideals are naturally associated to simple graphs and are called cover ideals. The reason for this naming is that the cover ideal $J(G)$ of a graph $G$ is minimally generated by squarefree monomials corresponding to the minimal vertex covers of $G$ (see Section \ref{2} for precise definition of cover ideals). Herzog and Hibi \cite[Theorem 3.6]{hh''} proved that if $G$ is a bipartite graph such that $J(G)$ has a linear resolution, then $J(G)^k$ has a linear resolution, for each integer $k\geq 1$. By \cite[Corollary 2.6]{grv}, for every bipartite graph $G$, we have $J(G)^{(k)}=J(G)^k$, where $J(G)^{(k)}$ denotes the $k$th symbolic power of $J(G)$. Thus, the result of Herzog and Hibi essentially says that if $G$ is a bipartite graph such that $J(G)$ has a linear resolution, then $J(G)^{(k)}$ has a linear resolution, for every integer $k\geq 1$. In
\cite[Theorem 3.6]{s3}, we generalized this result to very well-covered graphs. In the same paper, we posed the following question.

\begin{ques} [\cite{s3}, Page 105 and \cite{s8}, Question 3.10] \label{qwellres}
Let $G$ be a very well-covered graph and suppose that $J(G)^{(k)}$ has a linear resolution, for some integer $k\geq 2$. Is it true that $J(G)$ has a linear resolution?
\end{ques}

In \cite[Corollary 3.7]{s3}, we proved that Question \ref{qwellres} has a positive answer for the special class of bipartite graphs. As the first main result of this paper, we give a positive answer to this question, for any very well-covered graph (see Proposition \ref{linwell}). Indeed, we prove a stronger result. In Theorem \ref{main}, we characterize all graphs $G$ with the property that $J(G)^{(k)}$ has a linear resolution for some integer $k\geq 2$. It turns out that $J(G)^{(k)}$ has a linear resolution for some (equivalently, for all) integer $k\geq 2$ if and only if $G$ is a very well-covered graph such that $J(G)$ has a linear resolution.

For a homogenous ideal $I$, let ${\rm reg}(I)$ denote the Castelnuovo--Mumford regularity of $I$. The next main result of this paper is motivated by a result due to Mart${\rm\acute{i}}$nez-Bernal et al. \cite{mmvv}. In fact, it is shown in \cite[Corollary 5.3]{mmvv} that for any bipartite graph $G$ and for every integer $k\geq 1$, the inequality$${\rm reg}(J(G)^k)\leq {\rm reg}(J(G)^{k+1})$$holds. As we mentioned above, for any bipartite graph $G$ and for each integer $k\geq 1$, we have $J(G)^{(k)}=J(G)^k$. Hence, the above inequality means that
\[
\begin{array}{rl}
{\rm reg}(J(G)^{(k)})\leq {\rm reg}(J(G)^{(k+1)}).
\end{array} \tag{$\dagger$} \label{dag}
\]
when $G$ is a bipartite graph. In Theorem \ref{main2}, we prove that inequality \ref{dag} holds for any arbitrary graph $G$, generalizing \cite[Corollary 5.3]{mmvv}.

For a monomial ideal $I$, let ${\rm deg}(I)$ denote the maximum degree of minimal monomial generators of $I$. As the last result of this paper, we study ${\rm deg}(J(G)^{(k)})$. In Corollary \ref{deguc}, we compute ${\rm deg}(J(G)^{(k)})$ when $G$ is either an unmixed or a claw-free graph. As a consequence, we see that for any such graph, ${\rm deg}(J(G)^{(k)})$ is a a linear function of $k$. Note that in general, ${\rm deg}(J(G)^{(k)})$ is not a linear function (even eventually), as it is shown by Dung at al. \cite[Theorem 5.15]{dhnt}.


\section{Preliminaries} \label{sec2}

In this section, we provide the definitions and basic facts which will be used in the next sections.

Let $G$ be a simple graph with vertex set $V(G)=\big\{x_1, \ldots,
x_n\big\}$ and edge set $E(G)$. For a vertex $x_i$, the {\it neighbor set} of $x_i$ is $N_G(x_i)=\{x_j\mid \{x_i, x_j\}\in E(G)\}$ and we set $N_G[x_i]=N_G(x_i)\cup \{x_i\}$ and call it the {\it closed neighborhood} of $x_i$. For a subset $F\subseteq V(G)$, we set $N_G[F]=\cup_{x_i\in F}N_G[x_i]$. For every subset $A\subset V(G)$, the graph $G\setminus A$ is the graph with vertex set $V(G\setminus A)=V(G)\setminus A$ and edge set $E(G\setminus A)=\{e\in E(G)\mid e\cap A=\emptyset\}$. A subgraph $H$ of $G$ is called induced provided that two vertices of $H$ are adjacent if and only if they are adjacent in $G$. The graph $G$ is {\it bipartite} if there exists a partition $V(G)=A\cup B$ such
that each edge of $G$ is of the form $\{x_i,x_j\}$ with $x_i\in A$ and $x_j\in B$. If moreover, every vertex of $A$ is adjacent to every vertex of $B$, then we say that $G$ is a {\it complete bipartite} graph and denote it by $K_{a,b}$, where $a=|A|$ and $b=|B|$. The graph $K_{1,3}$ is called a {\it claw} and the graph $G$ is said to be {\it claw--free} if it has no claw as an induced subgraph. A subset $W$ of $V(G)$ is called an {\it independent subset} of $G$ if there are no edges among the vertices of $W$. An independent subset $W$ of $G$ is a {\it maximal independent subset}, if $W\cup\{x\}$ is not an independent subset of $G$, for every vertex $x\in V(G)\setminus W$. A subset $C$ of $V(G)$ is called a {\it vertex cover} of $G$ if every edge of $G$ is incident to at least one vertex of $C$. A vertex cover $C$ is called a {\it minimal vertex cover} of $G$ if no proper subset of $C$ is a vertex cover of $G$. Note that $C$ is a minimal
vertex cover if and only if $V(G)\setminus C$ is a maximal independent set. The graph $G$ is called {\it unmixed} if all
maximal independent subsets of $G$ have the same number of elements. A graph $G$ without isolated vertices is {\it very well-covered} if $n$ is an even number and every maximal independent subset of $G$ has cardinality $n/2$. In particular, any unmixed bipartite graphs without isolated vertices is a very well-covered graph. The following result from \cite{crt}, determines the structure of very well-covered graphs.

\begin{prop} \label{strwell} \cite[Proposition 2.3]{crt}
Let $G$ be a very well-covered graph with $2h$ vertices. Then the vertices of $G$ can be labeled as $V(G)=\{x_1, \ldots, x_h, y_1, \ldots, y_h\}$ such that the following conditions are satisfied.
\begin{itemize}
\item[(i)] $X=\{x_1, \ldots, x_h\}$ is a minimal vertex cover of $G$ and $Y=\{y_1, \ldots, y_h\}$ is a maximal independent subset of $G$.
\item[(ii)] $\{x_i, y_i\}\in E(G)$, for each integer $i=1, \ldots, h$.
\item[(iii)] If $\{z_i, x_j\}$, $\{y_j, x_k\}\in E(G)$, then $\{z_i, x_k\}\in E(G)$
for distinct indices $i$, $j$ and $k$ and for $z_i\in \{ x_i, y_i\}$.
\item[(iv)] If $\{x_i, y_j\} \in E(G)$, then $\{x_i, x_j\} \notin E(G)$.
\end{itemize}
\end{prop}

A {\it simplicial complex} $\Delta$ on the set of vertices $V(\Delta)=\{x_1,
\ldots, x_n\}$ is a collection of subsets of $V(\Delta)$ which is closed under
taking subsets; that is, if $F \in \Delta$ and $F'\subseteq F$, then also
$F'\in\Delta$. Every element $F\in\Delta$ is called a {\it face} of
$\Delta$, and its {\it dimension} is defined to be $\dim F=|F|-1$. The {\it dimension} of
$\Delta$ which is denoted by $\dim\Delta$, is $d-1$, where $d
=\max\{|F|\mid F\in\Delta\}$. A {\it facet} of $\Delta$ is a maximal face
of $\Delta$ with respect to inclusion. We say that $\Delta$ is {\it pure} if all facets
of $\Delta$ have the same dimension. A pure simplicial complex $\Delta$ is {\it strongly connected}
if for every pair of facets $F, G\in \Delta$, there is a sequence of facets $F=F_0, F_1, \ldots, F_m=G$ such that $\dim(F_i\cap F_{i+1})=\dim\Delta-1$,
for every integer $i$ with $0\leq i\leq m-1$. The {\it link} of $\Delta$ with respect to a face $F \in \Delta$ is the simplicial complex$${\rm lk_{\Delta}}F=\{G
\subseteq V(\Delta)\setminus F\mid G\cup F\in \Delta\}.$$

Let $S=\mathbb{K}[x_1, \dots, x_n]$ be the polynomial ring in $n$ variables over a field $\mathbb{K}$. For every
subset $F\subseteq V(\Delta)$, we set ${\it x}_F=\prod_{x_i\in F}x_i$. The {\it
Stanley--Reisner ideal of the simplicial complex $\Delta$ over $\mathbb{K}$} is the ideal $I_{
\Delta}$ of $S$ which is generated by those squarefree monomials $x_F$ with
$F\notin\Delta$. The {\it Stanley--Reisner ring of $\Delta$ over $\mathbb{K}$}, denoted by $\mathbb
{K}[\Delta]$, is defined to be $\mathbb{K} [\Delta]=S/I_{\Delta}$. The simplicial complex $\Delta$ is called {\it Cohen-Macaulay}, if its Stanley--Reisner ring is a Cohen-Macaulay ring.

Let $G$ be a graph. The {\it independence simplicial complex} of $G$ is defined by
$$\Delta(G)=\{A\subseteq V(G)\mid A \,\, \mbox{is an independent subset of}\,\,
G\}.$$It is easy to see that the Stanley--Reisner ideal of $\Delta(G)$ is the {\it edge ideal} of $G$ which is defined as$$I(G)=\big(x_ix_j\mid \{x_1, x_j\}\in E(G)\big)\subset S.$$ A graph $G$ is a {\it Cohen-Macaualy graph}, if $\Delta(G)$ is a Cohen-Macaulay simplicial complex. It is well-known that every Cohen-Macaulay graph is unmixed. The following result from \cite{mmcrty} provides a characterization for Cohen-Macaulay very well-covered graphs.

\begin{prop} \label{cmwell} \cite[Lemma 3.1]{mmcrty}
Let $G$ be a very well-covered graph with $2h$ vertices. Then $G$ is a Cohen-Macaulay graph if and only if the vertices of $G$ can be labeled as $V(G)=\{x_1, \ldots, x_h, y_1, \ldots, y_h\}$ such that conditions (i)-(iv) of Proposition \ref{strwell} are satisfied and moreover, $i\leq j$ whenever $\{x_i,y_j\}\in E(G)$.
\end{prop}

The Alexander dual of the edge ideal of $G$ in $S$, i.e., the
ideal $$J(G)=I(G)^{\vee}=\bigcap_{\{x_i,x_j\}\in E(G)}(x_i,x_j),$$ is called the
{\it cover ideal} of $G$. It is
well-known and easy to check that $J(G)$ is minimally generated by the monomials $\prod_{x_i\in C}x_i$, where $C$ is a minimal vertex cover of $G$. We refer to \cite{s8} for a survey about homological and combinatorial properties of powers of cover ideals.

\begin{dfn}
Let $I$ be an ideal of $S$ and let ${\rm Min}(I)$ be the set of minimal primes of $I$. For every integer $k\geq 0$, the $k$th {\it symbolic power} of $I$,
denoted by $I^{(k)}$, is defined to be$$I^{(k)}=\bigcap_{\frak{p}\in {\rm Min}(I)} {\rm Ker}(S\rightarrow (S/I^k)_{\frak{p}}).$$
\end{dfn}

Let $I$ be a squarefree monomial ideal in $S$ and suppose that $I$ has the irredundant
primary decomposition $$I=\frak{p}_1\cap\ldots\cap\frak{p}_r,$$ where every
$\frak{p}_i$ is an ideal of $S$ generated by a subset of the variables of
$S$. It follows from \cite[Proposition 1.4.4]{hh} that for every integer $k\geq 0$, $$I^{(k)}=\frak{p}_1^k\cap\ldots\cap
\frak{p}_r^k.$$In particular, for every graph $G$ and for any integer $k\geq 0$, we have$$J(G)^{(k)}=\bigcap_{\{x_i,x_j\}\in E(G)}(x_i,x_j)^k.$$

Assume that $M$ is a graded $S$-module and let
$$0 \longrightarrow \cdots \longrightarrow  \bigoplus_{j}S(-j)^{\beta_{1,j}(M)} \longrightarrow \bigoplus_{j}S(-j)^{\beta_{0,j}(M)} \longrightarrow M \longrightarrow 0$$be the minimal graded free resolution of $M$.
The integers $\beta_{i,j}(M)$ are called the graded Betti numbers of $M$. The Castelnuovo--Mumford regularity (or simply, regularity) of $M$, denote by ${\rm reg}(M)$, is defined as follows:$${\rm reg}(M)=\max\{j-i|\ \beta_{i,j}(M)\neq0\}.$$Moreover,$${\rm pd}(M)=\max\{i|\ \beta_{i,j}(M)\neq 0, {\rm \ for \ some} \ j\}$$is called the {\it projective dimension} of $M$.

A homogenous ideal $I$ is said to have a {\it linear resolution}, if for some integer $d$, $\beta_{i,i+t}(I)=0$
for all $i$ and for every integer $t\neq d$. It is clear from the definition that if an ideal has a linear resolution, then all the minimal generators of $I$ have the same degree. It follows from the Eagon-Reiner's theorem \cite[Theorem 8.1.9]{hh} that a graph $G$ is Cohen-Macaualy if and only if its cover ideal $J(G)$ has a linear resolution. Let $I$ be a homogenous ideal which is generated in a single degree $d$. We say that $I$ has a {\it linear presentation}, if $\beta_{1,1+t}(I)=0$ for each integer $t\neq d$. It is obvious that $I$ has a linear presentation if it has a linear resolution.

Let $I$ be a monomial ideal of
$S=\mathbb{K}[x_1,\ldots,x_n]$ with minimal generators $u_1,\ldots,u_m$,
where $u_j=\prod_{i=1}^{n}x_i^{a_{i,j}}$, $1\leq j\leq m$. For every $i$
with $1\leq i\leq n$, let$$a_i:=\max\{a_{i,j}\mid 1\leq j\leq m\},$$and
suppose that $$T=\mathbb{K}[x_{1,1},x_{1,2},\ldots,x_{1,a_1},x_{2,1},
x_{2,2},\ldots,x_{2,a_2},\ldots,x_{n,1},x_{n,2},\ldots,x_{n,a_n}]$$ is a
polynomial ring over the field $\mathbb{K}$. Let $I^{{\rm pol}}$ be the squarefree
monomial ideal of $T$ with minimal generators $u_1^{{\rm pol}},\ldots,u_m^{{\rm pol}}$, where
$u_j^{{\rm pol}}=\prod_{i=1}^{n}\prod_{k=1}^{a_{i,j}}x_{i,k}$, $1\leq j\leq m$. The
monomial $u_j^{{\rm pol}}$ is called the {\it polarization} of $u_j$, and the ideal $I^{{\rm pol}}$
is called the {\it polarization} of $I$. It is known that $\beta_{i,j}(I)=\beta_{i,j}(I^{{\rm pol}})$, for each pair of integers $i$ and $j$ (see e.g., \cite[Corollary 1.6.3]{hh}).


\section{Symbolic powers with linear resolution} \label{sec3}

Our goal in this section is to characterize all graphs $G$ with the property that $J(G)^{(k)}$ has a linear resolution for some integer $k\geq 2$. To achieve this goal, we first give a positive answer to Question \ref{qwellres}.

\begin{prop} \label{linwell}
Let $G$ be a very well-covered graph and suppose that $J(G)^{(k)}$ has a linear resolution, for some integer $k\geq 1$. Then $J(G)$ has a linear resolution.
\end{prop}

\begin{proof}
Using Proposition \ref{strwell}, the vertices of $G$ can be labeled as$$V(G)=\{x_1, \ldots, x_h, y_1, \ldots, y_h\}$$such that the following conditions are satisfied.
\begin{itemize}
\item[(i)] $X:=\{x_1, \ldots, x_h\}$ is a minimal vertex cover of $G$ and $Y:=\{y_1, \ldots, y_h\}$ is a maximal independent subset of $G$.
\item[(ii)] $\{x_i, y_i\}\in E(G)$, for each integer $i=1, \ldots, h$.
\item[(iii)] If $\{z_i, x_j\}$, $\{y_j, x_k\}\in E(G)$, then $\{z_i, x_k\}\in E(G)$
for distinct indices $i$, $j$ and $k$ and for $z_i\in \{ x_i, y_i\}$.
\item[(iv)] If $\{x_i, y_j\} \in E(G)$, then $\{x_i, x_j\} \notin E(G)$.
\end{itemize}

We know from \cite[Lemma 3.4]{s3} that $(J(G)^{(k)})^{{\rm pol}}$ is the cover ideal of a graph $G_k$ with vertex set $$V(G_k)=\big\{x_{i,p}, y_{i,p}\mid 1\leq i\leq h \ {\rm and} \  1\leq p\leq k\big\},$$ and edge set
\begin{align*}
E(G_k) & =\big\{\{x_{i,p}, x_{j,q}\}\mid \{x_i, x_j\}\in E(G) \  {\rm and} \  p+q\leq k+1\big\}\\
& \cup \big\{\{x_{i,p}, y_{j,q}\}\mid \{x_i, y_j\}\in E(G) \  {\rm and} \  p+q\leq k+1\big\}.
\end{align*}
It follows from the assumption and \cite[Corollary 1.6.3]{hh} that $J(G_k)$ has a linear resolution. As a consequence, \cite[Theorem 8.1.9]{hh} implies that $G_k$ is a Cohen--Macaulay graph. Note that$$F=\{y_{i,j}\mid 1\leq i\leq h \ {\rm and} \ 2\leq j\leq k\}$$is an independent subset of $G_k$. It is obvious that$$N_{G_k}[F]=F\cup\{x_{i,j}\mid 1\leq i\leq h \ {\rm and} \ 1\leq j\leq k-1\}.$$Thus, $G_k\setminus N_{G_k}[F]$ is isomorphic to the bipartite graph which is obtained from $G$ by deleting the edges whose endpoints are in $X$. We denote this new graph by $G'$. In other words, $G'$ is the graph with vertex set $V(G')=V(G)$ and edge set$$E(G')=\big\{\{x_i, y_j\} \mid 1\leq i,j\leq h \ {\rm and} \ \{x_i, y_j\}\in E(G)\big\}.$$It follows that ${\rm lk}_{\Delta(G_k)}F=\Delta(G')$. Since $G_k$ is Cohen--Macaulay, we conclude that $G'$ is Cohen--Macaulay too. Therefore, $G'$ is a very well-covered graph. Using \cite[Lemma 3.5]{crt}, there exists a permutation $\sigma : [h]\rightarrow [h]$ with the property that $\sigma(i)\leq \sigma(j)$ whenever $\{x_{\sigma(i)}, y_{\sigma(j)}\}\in E(G)$. One can easily see that the graph $G$ with labeling $x_{\sigma(1)}, \ldots, x_{\sigma(h)}, y_{\sigma(1)}, \ldots, y_{\sigma(h)}$ of its vertices also satisfies conditions (i)-(iv) of Proposition \ref{strwell}. Hence, using Proposition \ref{cmwell}, we conclude that $G$ is a Cohen-Macaulay graph. As a consequence, \cite[Theorem 8.1.9]{hh} implies that $J(G)$ has a linear resolution.
\end{proof}

As we mentioned in Section \ref{sec2}, if a homogenous ideal $I$ has a linear resolution, the it must be generated in a single degree. Thus, in order to characterize the graphs $G$ with the property that $J(G)^{(k)}$ has a linear resolution, we should first determine the graphs $G$ such that $J(G)^{(k)}$ is generated in a single degree. This will be done in the next lemma. We recall that for a monomial ideal $I$, the unique set of minimal monomial generators of $I$ is denoted by $G(I)$.

\begin{lem} \label{singdeg}
Let $G$ be a graph without isolated vertices and assume that $J(G)^{(k)}$ is generated in a single degree, for some integer $k\geq 2$. Then $G$ is a very well-covered graph.
\end{lem}

\begin{proof}
Assume that $V(G)=\{x_1, \ldots, x_n\}$ and let $u$ be a monomial in the set of minimal monomial generators of $J(G)$. Note that $u^k\in J(G)^k\subseteq J(G)^{(k)}$. As $u\in G(J(G))$, for every integer $1\leq i\leq n$, there exists a vertex $x_j\in N_G(x_i)$ such that $u/x_i\notin (x_i, x_j)$. Hence,$$u^k/x_i=(u/x_i)^kx_i^{k-1}\notin (x_i, x_j)^k.$$In particular, for every integer $i$ with $1\leq i\leq n$ we have $u^k/x_i\notin J(G)^{(k)}$. Therefore, $u^k$ belongs to the set of minimal monomial generators of $J(G)^{(k)}$. As $J(G)^{(k)}$ is generated in a single degree, we deduce that$${\rm deg}(J(G)^{(k)})={\rm deg}(u^k)=k{\rm deg}(u).$$Since the above equalities hold for every monomial $u\in G(J(G))$, we conclude that $J(G)$ is generated in a single degree. Hence, $G$ is an unmixed graph. Let $d$ denote the degree of minimal monomial generators of $J(G)$. It follows that$${\rm deg}(J(G)^{(k)})=kd.$$

By \cite{gv} (see also \cite[Theorem 0.1]{crt}), we have $d\geq n/2$. Thus, to complete the proof, we must show that $d\leq n/2$. Set $v:=x_1\ldots x_n$. We divide the rest of the proof in two cases.

\vspace{0.3cm}
{\bf Case 1.} Assume that $k=2m$ is an even integer. It is clear that $v^m\in J(G)^{(k)}$. This implies that$$mn={\rm deg}(v^m)\geq {\rm deg}(J(G)^{(k)})=kd=2md.$$Therefore, $n\geq 2d$.

\vspace{0.3cm}
{\bf Case 2.} Assume that $k=2m+1$ is an odd integer. As above, let $u$ be a monomial in the set of minimal monomial generators of $J(G)$. Obviously, $uv^m\in J(G)^{(k)}$. Hence,$$d+mn={\rm deg}(uv^m)\geq {\rm deg}(J(G)^{(k)})=kd=(2m+1)d$$which means that $n\geq 2d$.
\end{proof}

It is obvious from the definitions that the condition of having a linear resolution is stronger that having a linear presentation. However, we will see in Proposition \ref{linpre} that these two concepts are equivalent for symbolic powers of cover ideals of very well-covered graphs. In order to prove this proposition, we need to recall the notion of Serre's condition.

Let $M$ be a nonzero finitely generated $S$-module and let $r\geq 1$ be a positive integer. Then $M$ is said to satisfy the {\it Serre's condition} ($S_r$), if for every prime ideal $\frak{p}$ of $S$, the inequality$${\rm depth}\ M_{\frak{p}}\geq \min\{r,\dim M_{\frak
{p}}\}$$holds true. We say that a simplicial complex $\Delta$ satisfies the Serre's condition ($S_r$) if its Stanley--Reisner ring satisfies the same condition. We refer to \cite{psty} for a survey on simplicial complexes satisfying the Serre's condition.

\begin{prop} \label{linpre}
Let $G$ be a very well-covered graph and assume that $k$ is a positive integer. Then $J(G)^{(k)}$ has a linear resolution if and only if $J(G)^{(k)}$ has a linear presentation.
\end{prop}

\begin{proof}
The "only if" part is obvious. Thus, we only prove the "if" part. So suppose that $J(G)^{(k)}$ has a linear presentation. By \cite[Corollary 1.6.3]{hh}, the ideal $(J(G)^{(k)})^{{\rm pol}}$ has a linear presentation. Using \cite[Proposition 3.1 and Lemma 3.4]{s3}, there is a very well-covered graph $G_k$ with $(J(G)^{(k)})^{{\rm pol}}=J(G_k)$. As $J(G_k)$ has a linear presentation, it follows from \cite[Corollary 3.7]{y} that $\Delta(G_k)$ satisfies the Serre's condition $(S_2)$. It is well-known that every simplicial complex satisfying the Serre's condition $(S_2)$ is strongly connected (see e.g. \cite[Page 2]{mt'}). In particular, $\Delta(G_k)$ is a strongly connected simplicial complex. Hence, we conclude from \cite[Theorem 0.2]{crt} that $G_k$ is a Cohen-Macaulay graph. Therefore, \cite[Theorem 8.1.9]{hh} implies that $J(G_k)$ has a linear resolution. Consequently, it follows from \cite[Corollary 1.6.3]{hh} that $J(G)^{(k)}$ has a linear resolution.
\end{proof}

We are now ready to prove the main result of this section.

\begin{thm} \label{main}
For any graph $G$ without isolated vertices, the following conditions are equivalent.
\begin{itemize}
\item[(i)] $J(G)^{(k)}$ has a linear resolution, for every integer $k\geq 1$.
\item[(ii)] $J(G)^{(k)}$ has a linear resolution, for some integer $k\geq 2$.
\item[(iii)] $J(G)^{(k)}$ has a linear presentation, for every integer $k\geq 1$.
\item[(iv)] $J(G)^{(k)}$ has a linear presentation, for some integer $k\geq 2$.
\item[(v)] $G$ is a Cohen-Macaulay very well-covered graph.
\end{itemize}
\end{thm}

\begin{proof}
If $G$ satisfies either of the conditions (i)-(iv), then $J(G)$ is generated in a single degree. Hence, using Lemma \ref{singdeg}, we deduce that $G$ is a very well-covered graph. Thus, the equivalences (i)$\Leftrightarrow$(iii) and (ii)$\Leftrightarrow$(iv) follow from Proposition \ref{linpre}. The implication (i)$\Rightarrow$(ii) is trivial. By \cite[Theorem 8.1.9]{hh} and \cite[Theorem 3.6]{s3}, (v) implies (i). To prove (ii)$\Rightarrow$ (v), as we mentioned above, it follows from (ii) and Lemma \ref{singdeg} that $G$ is a very well-covered graph. Then using Proposition \ref{linwell} and \cite[Theorem 8.1.9]{hh}, we conclude that $G$ is a Cohen-Macaulay graph.
\end{proof}

\begin{rems}
\begin{enumerate}
\item Assume that $I$ is a squarefree monomial ideal. In general, it is not true that if $I^{(k)}$ has a linear resolution, for some integer $k\geq 2$, then $I$ has the same property. For instance, let $I=(x_1x_2, x_2x_3, x_3x_4, x_4x_5, x_1x_5)$ be the edge ideal of the $5$-cycle graph. On can easily check that $I^{(2)}$ has a liner resolution but $I$ has not.
\item Assume that $I$ is a squarefree monomial ideal. In general, it is not true that if $I^{(2)}$ has a linear resolution then $I^{(k)}$ has the same property, for every integer $k\geq 2$. For example, let $I$ be the same ideal as above. Then $I^{(2)}$ has a liner resolution. On the other hand, $x_1^3x_2^3$ and $x_1x_2x_3x_4x_5$ belong to set of minimal monomial generators of $I^{(3)}$. In particular, $I^{(3)}$ is not generated in a single degree. Hence, it has not a linear resolution.
\end{enumerate}
\end{rems}


\section{Regularity of symbolic powers} \label{sec4}

In this section, we study the regularity function of symbolic powers of cover ideals of graphs. Mart${\rm\acute{i}}$nez-Bernal et al. \cite[Corollary 5.3]{mmvv} proved that the regularity of powers of cover ideals of every bipartite graph is a nondecreasing function. In Theorem \ref{main2}, we generalize this result by showing that the sequence $\big({\rm reg}(J(G)^{(k)})\big)_{k=1}^{\infty}$ is nondecreasing, for any arbitrary graph $G$.

\begin{thm} \label{main2}
Let $G$ be a graph. Then for every integer $k\geq 1$, we have$${\rm reg}(S/J(G)^{(k)})\leq {\rm reg}(S/J(G)^{(k+1)}).$$
\end{thm}

\begin{proof}
Without lose of generality, we may suppose that $G$ has no isolated vertex. Let $V(G)=\{x_1, \ldots, x_n\}$ be the vertex set of $G$. Consider the squarefree monomial ideals $(J(G)^{(k)})^{{\rm pol}}\subseteq T_k$ and $(J(G)^{(k+1)})^{{\rm pol}}\subseteq T_{k+1}$, where$$T_r=\mathbb{K}[x_{i,p}\mid 1\leq i\leq n,  1\leq p\leq r],$$for each $r=k, k+1$. It follows from \cite[Corollary 1.6.3]{hh} that$${\rm reg}\big(T_r/(J(G)^{(r)})^{{\rm pol}}\big)={\rm reg}\big(S/J(G)^{(r)}\big).$$As a consequence, it is enough to prove that$${\rm reg}\big(T_k/(J(G)^{(k)})^{{\rm pol}}\big)\leq {\rm reg}\big(T_{k+1}/(J(G)^{(k+1)})^{{\rm pol}}\big).$$

By \cite[Lemma 3.4]{s3}, for $r=k, k+1$, the ideal $(J(G)^{(r)})^{{\rm pol}}$ is the cover ideal of a graph $G_r$, with vertex set$$V(G_r)=\{x_{i,p}\mid 1\leq i\leq n,  1\leq p\leq r\}$$and edge set$$E(G_r)=\{\{x_{i,p}, x_{j,q}\}\mid \{x_i, x_j\}\in E(G) \  {\rm and} \  p+q\leq r+1\}.$$Set$$W:=\{x_{i, \lfloor\frac{k+1}{2}\rfloor +1} \mid 1\leq i\leq n\}\subseteq V(G_{k+1}),$$and consider the map $\varphi : V(G_k)\rightarrow V(G_{k+1}\setminus W)$ which is defined as follows.

$$\varphi(x_{i,j}) =
\left\{
	\begin{array}{ll}
		x_{i,j}  & \mbox{if } 1\leq j\leq \lfloor\frac{k+1}{2}\rfloor \\
		x_{i,j+1} & \mbox{if } j\geq  \lfloor\frac{k+1}{2}\rfloor+1
	\end{array}
\right.$$

It is easy to see that $\varphi$ is an isomorphism between $G_k$ and $G_{k+1}\setminus W$. Hence, $G_k$ is an induced subgraph of $G_{k+1}$. It follows from \cite[Proposition 4.1.1]{j} that$${\rm pd}(I(G_k))\leq {\rm pd}(I(G_{k+1})).$$Using Terai's theorem \cite[Proposition 8.1.10]{hh}, we deduce that$${\rm reg}\big(T_k/(J(G_k)\big)\leq {\rm reg}\big(T_{k+1}/(J(G_{k+1})\big)$$which means that$${\rm reg}\big(T_k/(J(G)^{(k)})^{{\rm pol}}\big)\leq {\rm reg}\big(T_{k+1}/(J(G)^{(k+1)})^{{\rm pol}}\big).$$This completes the proof.
\end{proof}


\section{Maximum degree of generators of symbolic powers} \label{sec5}

As the last result of this paper, we study the maximum degree of generators of symbolic powers of cover ideals. Thanks to \cite[Theorem 5.15]{dhnt}, we know that ${\rm deg}(J(G)^{(k)})$ is not necessarily a linear function (even eventually). However, we will see in Theorem \ref{deg} and Corollary \ref{deguc} that ${\rm deg}(J(G)^{(k)})$ is a linear function, for interesting classes of graphs. In order to prove theses results, we need the following lemma.

\begin{lem} \label{del}
Let $G$ ba a graph with vertex set $V(G)=\{x_1, \ldots, x_n\}$. Assume that $x\in V(G)$ is a vertex of $G$. Set $u:=\prod_{x_i\in N_G(x)}x_i$ and $J'=J(G\setminus N_G[x])S$. Then $J(G)^{(k)}+(x)=u^kJ'^{(k)}+(x)$.
\end{lem}

\begin{proof}
We have
\begin{align*}
& J(G)^{(k)}+(x)=\bigcap_{\{x_i,x_j\}\in E(G)}(x_i,x_j)^k+(x)=\bigcap_{\{x_i,x_j\}\in E(G)}\big((x_i,x_j)^k+(x)\big)\\ & =\bigg(\bigcap_{x_i\in N_G(x)}\big((x_i,x)^k+(x)\big)\bigg) \cap \bigg(\bigcap_{\substack{\{x_i,x_j\}\in E(G) \\ x\notin \{x_i,x_j\}}}\big((x_i,x_j)^k+(x)\big)\bigg)\\ & =\bigg(\bigcap_{x_i\in N_G(x)}\big((x_i)^k+(x)\big)\bigg) \cap \bigg(\bigcap_{\substack{\{x_i,x_j\}\in E(G) \\ x\notin \{x_i,x_j\}}}\big((x_i,x_j)^k+(x)\big)\bigg).
\end{align*}

Note that for every vertex $x_i\in N_G(x)$ and for every edge $\{x_i,x_j\}\in E(G)$, we have $(x_i)^k+(x)\subseteq (x_i,x_j)^k+(x)$. Thus, it follows from the above equalities that
\begin{align*}
& J(G)^{(k)}+(x)=\bigg(\bigcap_{x_i\in N_G(x)}\big((x_i)^k+(x)\big)\bigg) \cap \bigg(\bigcap_{\{x_i,x_j\}\in E(G\setminus N_G[x])}\big((x_i,x_j)^k+(x)\big)\bigg)\\ & =\bigg(\bigcap_{x_i\in N_G(x)}(x_i)^k \cap \bigcap_{\{x_i,x_j\}\in E(G\setminus N_G[x])}(x_i,x_j)^k\bigg)+(x)\\ & =\big(u^k\cap J'^{(k)}\big)+(x)=u^kJ'^{(k)}+(x).
\end{align*}
\end{proof}

We are now ready to prove the last result of this paper.

\begin{thm} \label{deg}
Let $\mathcal{H}$ be a family of graphs which satisfies the following conditions.
\begin{itemize}
\item[(i)] For every graph $G\in \mathcal{H}$ and every vertex $x\in V(G)$, the graph $G\setminus N_G[x]$ belongs to $\mathcal{H}$.
\item[(ii)] If $G\in \mathcal{H}$ has no isolated vertex, then it admits a minimal vertex cover with cardinality at least $\frac{|V(G)|}{2}$.
\end{itemize}
Then for every graph $G\in \mathcal{H}$ and every integer $k\geq 1$, we have$${\rm deg}(J(G)^{(k)})=k{\rm deg}(J(G)).$$
\end{thm}

\begin{proof}
By replacing $\mathcal{H}$ with $\mathcal{H}\cup\{K_2\}$, we may assume that $K_2\in \mathcal{H}$. Let $m$ be the number of edges of $G$. We use induction on $k+m$. For $k=1$, there is nothing to prove. Therefore, assume that $k\geq 2$. Suppose $m=1$ and let $\{x_1, x_2\}$ be the unique edge of $G$. Then $J(G)=(x_1,x_2)$. Hence, ${\rm deg}(J(G))=1$ and ${\rm deg}(J(G)^{(k)})=k$. Thus, assume that $m\geq 2$. Without lose of generality, we may suppose that $G$ has no isolated vertex. By \cite[Lemma 3.1]{s7}, we have ${\rm deg}(J(G)^{(k)})\geq k{\rm deg}(J(G))$. Therefore, it is enough to prove that for every monomial $w\in G(J(G)^{(k)})$, the inequality ${\rm deg}(w)\leq k{\rm deg}(J(G))$ holds. Assume that $V(G)=\{x_1, \ldots, x_n\}$ and set $v:=x_1\ldots x_n$. We consider the following two cases.

\vspace{0.3cm}
{\bf Case 1.} Suppose $v$ divides $w$. Then $w/v$ belongs to the set of minimal monomial generators of $(J(G)^{(k)}:v)$. We know from \cite[Lemma 3.4]{s6} that $$(J(G)^{(k)}:v)=J(G)^{(k-2)}.$$Hence, using the induction hypothesis, we deduce that$${\rm deg}(w/v)\leq (k-2){\rm deg}(J(G)).$$Consequently,$${\rm deg}(w)\leq (k-2){\rm deg}(J(G))+n.$$ On the other hand, by assumption we have $n\leq 2{\rm deg}(J(G))$. Thus, the above inequality implies that ${\rm deg}(w)\leq k{\rm deg}(J(G))$.

\vspace{0.3cm}
{\bf Case 2.} Suppose $v$ does not divide $w$. Then there is a variable, say $x_j$ which does not divide $w$. Set $u:=\prod_{x_i\in N_G(x_j)}x_i$. We conclude from Lemma \ref{del} that $w$ belongs to the set of minimal monomial generators of $u^kJ(G\setminus N_G[x_j])^{(k)}$. In other words, $u^k$ divides $w$, and$$w/u^k\in G(J(G\setminus N_G[x_j])^{(k)}).$$It follows from the induction hypothesis that
\[
\begin{array}{rl}
{\rm deg}(w)-k{\rm deg}_G(x_j)={\rm deg}(w)-k{\rm deg}(u)\leq k{\rm deg}(J(G\setminus N_G[x_j]).
\end{array} \tag{1} \label{1}
\]
On the other hand, if $C$ is a minimal vertex cover of $G\setminus N_G[x_j]$, then $C\cup N_G(x_j)$ is a minimal vertex cover of $G$. Therefore,
\[
\begin{array}{rl}
{\rm deg}(J(G\setminus N_G[x_j]))+{\rm deg}_G(x_j)\leq {\rm deg}(J(G)).
\end{array} \tag{2} \label{2}
\]

Using (\ref{1}) and (\ref{2}), we deduce that$${\rm deg}(w)\leq k{\rm deg}(J(G\setminus N_G[x_j])+k{\rm deg}_G(x_j)\leq k{\rm deg}(J(G)).$$
\end{proof}

As a consequence of Theorem \ref{deg}, we obtain the following corollary.

\begin{cor} \label{deguc}
Let $G$ be a graph which is either unmixed or claw-free. Then for every integer $k\geq 1$, we have$${\rm deg}(J(G)^{(k)})=k{\rm deg}(J(G)).$$In particular, ${\rm deg}(J(G)^{(k)})$ is a linear function of $k$.
\end{cor}

\begin{proof}
By \cite{gv} (see also \cite[Theorem 0.1]{crt}), every unmixed graph $G$ without isolated vertices has a minimal vertex cover with cardinality at least $\frac{|V(G)|}{2}$. Also, we know from the proof of \cite[Theorem 3.7]{s7} that every claw-free graph $G$ without isolated vertices has a minimal vertex cover with cardinality at least $\frac{|V(G)|}{2}$. Let $\mathcal{H}$ denote the class of graphs which are either unmixed or claw-free. Then the assertion follows from Theorem \ref{deg}.
\end{proof}





\end{document}